\documentclass{article}%
\usepackage{amsmath}
\usepackage{amsfonts}
\usepackage{amssymb}
\usepackage{graphicx}%
\providecommand{\U}[1]{\protect \rule{.1in}{.1in}}
\newtheorem{theorem}{Theorem}[section]

\newtheorem{corollary}[theorem]{Corollary}

\newtheorem{definition}[theorem]{Definition}
\newtheorem{example}[theorem]{Example}

\newtheorem{lemma}[theorem]{Lemma}

\newtheorem{proposition}[theorem]{Proposition}
\newtheorem{remark}[theorem]{Remark}

\newenvironment{proof}[1][Proof]{\noindent \textbf{#1.} }{\  \rule{0.5em}{0.5em}}

\begin{document}

\title{Stopping Times and Related It\^{o}'s Calculus with $G$-Brownian Motion}
\author{Xinpeng LI and Shige PENG \thanks{Partially supported by the National Basic
Research Program of China (973 Program) grant No. 2007CB814900 (Financial
Risk).}\\School of Mathematics, Shandong University \\250100, Jinan, China }
\date{October 17, 2009}
\maketitle

\begin{quotation}
\textbf{Abstract. }Under the framework of $G$-expectation and $G$-Brownian
motion, We have introduced a It\^{o}'s integral for stochastic processes
without the condition of quasi-continuous. We then can obtain It\^{o}'s
integral on stopping time interval. This formulation help us to obtain
It\^{o}'s formula for a general $C^{1,2}$-function, which generalizes the
previous results of Peng \cite{Peng2006a,Peng2006b,Peng2007} and it's improved
version of Gao \cite{Gao2009}.
\end{quotation}

\section{Introduction}

$G$-Brownian motion is a continuous process $(B_{t})_{t\geq0}$ defined on a
sublinear expectation space $(\Omega,\mathcal{H},\mathbb{\hat{E}})$ (see
Definition \ref{Def-1}) with stable and independent increments. It was proved
that each increment $X=B_{t+s}-B_{t}$ of $B$ is $G$-normal distributed,
namely
\[
aX+b\bar{X}\overset{d}{=}\sqrt{a^{2}+b^{2}}X\text{,\  \ for }a,b\geq0,
\]
where $\bar{X}$ is an independent copy of $X$. A new type of stochastic
integral and the related It\^{o}'s calculus has been introduced in
\cite{Peng2006a,Peng2006b,Peng2007}. For example, if $\varphi$ is a $C^{2}%
$-function such that $\varphi_{xx}(x)$ satisfies polynomial growth function,
then we have
\begin{equation}
\varphi(B_{t})-\varphi(B_{t_{0}})=\int_{t_{0}}^{t}\varphi_{x}(B_{s}%
)dB_{s}+\frac{1}{2}\int_{0}^{t}\varphi_{xx}(B_{s})d\left \langle B\right \rangle
_{s}.\label{eq0}%
\end{equation}
A interesting problem is how to extend the above formulation to the situation
where $\varphi$ is simply a $C^{2}$-function. The main obstacle to treat this
situation is that the notion of stopping times and the related properties have
not yet been well-understood and studied within the framework of
$G$-expectation and $G$-Brownian motions. A difficulty hidden behind is that
uttill now the theory is mainly based on the space of random variables
$X=X(\omega)$ which are quasi-continuous with respect to the natural Choquet
capacity $\hat{c}(A):=\mathbb{\hat{E}}[\mathbf{I}_{A}]$, $A\in \mathcal{B}%
(\Omega)$. It is not yet clear that the martingale properties still hold for
random variables without quasi-continuous condition. On the other hand,
stopping times are closely related to random variables without
quasi-continuous properties. Recently Gao \cite{Gao2009} has improved the
It\^{o}'s formula of Peng. But the problem of (\ref{eq0}) for $C^{2}$-function
is still open.

In this paper we will face this difficulty by introducing It\^{o}'s stochastic
integrals $\int_{0}^{t}\eta_{s}dB_{s}$ where, for each $t$, the integrand
$\eta_{t}$ needs not to be a quasi-continuous random variable. Within this
framework we can treat a fundamentally important It\^{o}'s integral $\int
_{0}^{t\wedge \tau}\eta_{s}dB_{s}$ for a stopping time $\tau$ and then obtain
some important properties for the related stochastic calculus. A very general
form of It\^{o}'s formula with respect to $G$-Brownian motion has been
obtained. In particular (\ref{eq0}) is proved to be true for $\varphi \in
C^{2}$. Many important and interesting problems still open under this new
framework, e.g., under what condition $\int_{0}^{\cdot}\eta_{s}dB$ is a
martingale or a local martingale?

This paper is organized as follows: In the next section we recall some basic
notions an results of $G$-Brownian motion under a $G$-expectation and the
related space of random variables. In Section 3 we introduce a new space
$M_{\ast}^{2}(0,T)$ of stochastic processes which are not necessarily
quasi-continuous and then define the related It\^{o}'s integral on this space.
In Section 4 we discuss It\^{o}'s integral defined on $[0,\tau]$ where $\tau$
is a stopping time. This allows us to have a It\^{o}'s integral for a space
larger that $M_{\ast}^{2}(0,T)$. Finally in Section 5, we prove the mentioned
general form of It\^{o}'s formula.

We believe that some notions and properties of this papper will become
important and basic tools in the further development of $G$-Brownian motion
and the corresponding nonlinear expectation analysis.   

\section{Basic settings}

We present some preliminaries in the theory of sublinear expectations and the
related $G$-Brownian motions. More details can be found in Peng
\cite{Peng2006a}, \cite{Peng2006b} and \cite{Peng2007}.

\begin{definition}
{\label{Def-1} { Let }}$\Omega$ be a given set and let $\mathcal{H}$ be a
linear space of real valued functions defined on $\Omega$ with $c\in
\mathcal{H}$ for all constants $c$, and $|X|\in \mathcal{H}$, if $X\in
\mathcal{H}$. $\mathcal{H}$ is considered as the space of our
\textquotedblleft random variables\textquotedblright. {{A \textbf{sublinear
expectation }$\mathbb{\hat{E}}$ on $\mathcal{H}$ is a functional
$\mathbb{\hat{E}}:\mathcal{H}\mapsto \mathbb{R}$ satisfying the following
properties: for all $X,Y\in \mathcal{H}$, we have\newline \  \  \newline%
\textbf{(a) Monotonicity:} \  \  \  \  \  \  \  \  \  \  \  \  \ If $X\geq Y$ then
$\mathbb{\hat{E}}[X]\geq \mathbb{\hat{E}}[Y].$\newline \textbf{(b) Constant
preserving: \  \ }\  \ $\mathbb{\hat{E}}[c]=c$.\newline \textbf{(c)}
\textbf{Sub-additivity: \  \  \  \ }}}\  \  \  \  \  \  \  \ $\mathbb{\hat{E}%
}[X]-\mathbb{\hat{E}}[Y]\leq \mathbb{\hat{E}}[X-Y].$\newline{{\textbf{(d)
Positive homogeneity: } \ $\mathbb{\hat{E}}[\lambda X]=\lambda \mathbb{\hat{E}%
}[X]$,$\  \  \forall \lambda \geq0$.\newline}}\newline The triple $(\Omega
,\mathcal{H},\mathbb{\hat{E}})$ is called a sublinear expectation space.
$X\in \mathcal{H}$ is called a random variable in $(\Omega,\mathcal{H})$. We
often call $Y=(Y_{1},\cdots,Y_{d})$, $Y_{i}\in \mathcal{H}$ a $d$-dimensional
random vector in $(\Omega,\mathcal{H})$. Let us consider a space of random
variables $\mathcal{H}$ satisfying: if $X_{i}\in \mathcal{H}$, $i=1,\cdots,d$,
then%
\[
\varphi(X_{1},\cdots,X_{d})\in \mathcal{H}\text{,\  \ for all }\varphi \in
C_{b,Lip}(\mathbb{R}^{d}),
\]
where $C_{b,Lip}(\mathbb{R}^{d})$ is the space of all bounded and Lipschitz
continuous functions on $\mathbb{R}^{d}$. An $m$-dimensional random vector
$X=(X_{1},\cdots,X_{m})$ is said to be independent from another $n$%
-dimensional random vector $Y=(Y_{1},\cdots,Y_{n})$ if
\[
\mathbb{\hat{E}}[\varphi(X,Y)]=\mathbb{\hat{E}}[\mathbb{\hat{E}}%
[\varphi(X,y)]_{y=Y}],\  \  \text{for }\varphi \in C_{b,Lip}(\mathbb{R}^{m}%
\times \mathbb{R}^{n}).
\]
Let $X_{1}$ and $X_{2}$ be two $n$--dimensional random vectors defined
respectively in {sublinear expectation spaces }$(\Omega_{1},\mathcal{H}%
_{1},\mathbb{\hat{E}}_{1})${ and }$(\Omega_{2},\mathcal{H}_{2},\mathbb{\hat
{E}}_{2})$. They are called identically distributed, denoted by $X_{1}\sim
X_{2}$, if
\[
\mathbb{\hat{E}}_{1}[\varphi(X_{1})]=\mathbb{\hat{E}}_{2}[\varphi
(X_{2})],\  \  \  \forall \varphi \in C_{b.Lip}(\mathbb{R}^{n}).
\]
If $X$, $\bar{X}$ are two $m$-dimensional random vectors in $(\Omega
,\mathcal{H},\mathbb{\hat{E}})$ and $\bar{X}$ is identically distributed with
$X$ and independent from $X$, then $\bar{X}$ is said to be an independent copy
of $X$.
\end{definition}

\begin{definition}
(\textbf{$G$-normal distribution}) \label{Def-Gnormal} A $d$-dimensional
random vector $X=(X_{1},\cdots,X_{d})$ in a sublinear expectation space
$(\Omega,\mathcal{H},\mathbb{\hat{E}})$ is called $G$-normal distributed if
for each $a\,$, $b\geq0$ we have
\begin{equation}
aX+b\bar{X}\sim \sqrt{a^{2}+b^{2}}X,\  \label{srt-a2b2}%
\end{equation}
where $\bar{X}$ is an independent copy of $X$. Here the letter $G$ denotes the
function
\[
G(A):=\frac{1}{2}\mathbb{\hat{E}}[(AX,X)]:\mathbb{S}_{d}\mapsto \mathbb{R}.
\]
It is also proved in Peng \cite{Peng2006b, Peng2007} that, for each
$\mathbf{a}\in \mathbb{R}^{d}$ and $p\in \lbrack1,\infty)$%
\[
\mathbb{\hat{E}}[|\left(  \mathbf{a},X\right)  |^{p}]=\frac{1}{\sqrt
{2\pi \sigma_{\mathbf{aa}^{T}}^{2}}}\int_{-\infty}^{\infty}|x|^{p}\exp \left(
\frac{-x^{2}}{2\sigma_{\mathbf{aa}^{T}}^{2}}\right)  dx,
\]
where $\sigma_{\mathbf{aa}^{T}}^{2}=2G(\mathbf{aa}^{T})$.
\end{definition}

\begin{definition}
A $d$-dimensional stochastic process $\xi_{t}(\omega)=(\xi_{t}^{1},\cdots
,\xi_{t}^{d})(\omega)$ defined in a sublinear expectation space $(\Omega
,\mathcal{H},\mathbb{\hat{E}})$ is a family of $d$-dimentional random vectors
$\xi_{t}$ parametrized by $t\in \lbrack0,\infty)$ such \ that $\xi_{t}^{i}%
\in \mathcal{H}$, for each $i=1,\cdots,d$ and $t\in \lbrack0,\infty)$. 
\end{definition}

The most typical stochastic process in a sublinear expectation space is the
so-called $G$-Brownian motion. 

\begin{definition}
\label{Def-3}(\textbf{\cite{Peng2006a} and \cite{Peng2007}}) Let
$G:\mathbb{S}_{d}\mapsto \mathbb{R}$ be a given monotonic and sublinear
function. A process $\{B_{t}(\omega)\}_{t\geq0}$ in a sublinear expectation
space $(\Omega,\mathcal{H},\mathbb{\hat{E}})$ is called a $G$%
\textbf{--Brownian motion} if for each $n\in \mathbb{N}$ and $0\leq
t_{1},\cdots,t_{n}<\infty$,$\ B_{t_{1}},\cdots,B_{t_{n}}\in \mathcal{H}$ and
the following properties are satisfied: \newline \textsl{(i)} $B_{0}(\omega
)=0$;\newline \textsl{(ii)} For each $t,s\geq0$, the increment $B_{t+s}-B_{t}$
is independent to $(B_{t_{1}},B_{t_{2}},\cdots,B_{t_{n}})$, for each
$n\in \mathbb{N}$ and $0\leq t_{1}\leq \cdots \leq t_{n}\leq t$; \newline%
\textsl{(iii)} $B_{t+s}-B_{t}\sim \sqrt{s}X$, for $s,t\geq0,$ where $X$ is
$G$-normal distributed.
\end{definition}


We denote:

$\Omega=C_{0}^{d}(\mathbb{R}^{+})$ the space of all $\mathbb{R}^{d}$-valued
continuous functions $(\omega_{t})_{t\in \mathbb{R}^{+}}$, with $\omega_{0}=0$,
equipped with the distance
\[
\rho(\omega^{1},\omega^{2}):=\sum_{i=1}^{\infty}2^{-i}[(\max_{t\in \lbrack
0,i]}|\omega_{t}^{1}-\omega_{t}^{2}|)\wedge1].
\]
$\mathcal{B}(\Omega)$ denotes the $\sigma$-algebra generated by all open sets.
Let $\Omega=C_{0}(\mathbb{R}^{+})$ be the space of all $\mathbb{R}$-valued
continuous paths $(\omega_{t})_{t\in \mathbb{R}^{+}}$ with $\omega_{0}=0$,
equipped with the distance%

\[
\rho(\omega^{1},\omega^{2}):=\sum_{i=1}^{\infty}2^{-i}[(\max_{t\in \lbrack
0,i]}|\omega_{t}^{1}-\omega_{t}^{2}|)\wedge1].
\]
We denote by $\mathcal{B}(\Omega)$ the Borel $\sigma$-algebra of $\Omega$ and
by $\mathcal{M}$ the collection of all probability measure on $(\Omega
,\mathcal{B}(\Omega))$.

We also denote, for each $t\in \lbrack0,\infty)$:

\begin{itemize}
\item $\Omega_{t}:=\{ \omega_{\cdot \wedge t}:\omega \in \Omega \}$,

\item $\mathcal{F}_{t}:=\mathcal{B}(\Omega_{t})$,

\item $L^{0}(\Omega)$: the space of all $\mathcal{B}(\Omega)$-measurable real functions,

\item $L^{0}(\Omega_{t})$: the space of all $\mathcal{B}(\Omega_{t}%
)$-measurable real functions,

\item $B_{b}(\Omega)$: all bounded elements in $L^{0}(\Omega)$, $B_{b}%
(\Omega_{t})):=B_{b}(\Omega)\cap L^{0}(\Omega_{t})$,

\item $C_{b}(\Omega)$: all continuous elements in $B_{b}(\Omega)$;
$C_{b}(\Omega_{t}):=B_{b}(\Omega)\cap L^{0}(\Omega_{t})$.
\end{itemize}

In \cite{Peng2006b, Peng2007}, a $G$-Brownian motion is constructed on a
sublinear expectation sapce $(\Omega,\mathbb{L}_{G}^{p}(\Omega),\mathbb{\hat
{E}})$ for $p\geq1$, with $\mathbb{L}_{G}^{p}(\Omega)$ such that
$\mathbb{L}_{G}^{p}(\Omega)$ is a Banach space under the natural norm
$\left \Vert X\right \Vert _{p}:=\mathbb{\hat{E}}[|X|^{p}]^{1/p}$. In this space
the corresponding canonical process $B_{t}(\omega)=\omega_{t}$, $t\in
\lbrack0,\infty)$, for $\omega \in \Omega$ is a $G$-Brownian motion.
Furthermore, it is proved in \cite{DHP} that $L^{0}(\Omega)\supset
\mathbb{L}_{G}^{p}(\Omega)\supset C_{b}(\Omega)$, and there exists a weakly
compact family $\mathcal{P}$ of probability measures defined on $(\Omega
,\mathcal{B}(\Omega))$ such that
\[
\mathbb{\hat{E}}[X]=\sup_{P\in \mathcal{P}}E_{P}[X],\  \text{for}\ X\in
C_{b}(\Omega).
\]
We introduce the natural choquet capacity%
\[
\hat{c}(A):=\sup_{P\in \mathcal{P}}P(A),\  \ A\in \mathcal{B}(\Omega).
\]
The space $\mathbb{L}_{G}^{2}(\Omega)$ was also intruduced independently in
\cite{Denis-M} in a quite different framework.

\begin{definition}
A set $A\subset \Omega$ is polar if $\hat{c}(A)=0$. A property holds
\textquotedblleft quasi-surely\textquotedblright \ (q.s.) if it holds outside a
polar set.
\end{definition}

$\mathbb{L}_{G}^{p}(\Omega)$ can be characterized as follows:
\[
\mathbb{L}_{G}^{p}(\Omega)=\{X\in \mathbb{L}^{0}(\Omega)|\sup_{P\in \mathcal{P}%
}E_{P}[|X|^{p}]<\infty \text{, and }X\text{ is }\hat{c}\text{-quasi surely
continuous}\}.\text{ }%
\]
We also denote, for $p>0$,

\begin{itemize}
\item $\mathcal{L}^{p}:=\{X\in L^{0}(\Omega):\mathbb{\hat{E}}[|X|^{p}%
]=\sup_{P\in \mathcal{P}}E_{P}[|X|^{p}]<\infty \}$;\ 

\item $\mathcal{N}^{p}:=\{X\in L^{0}(\Omega):\mathbb{\hat{E}}[|X|^{p}]=0\}$;

\item $\mathcal{N}:=\{X\in L^{0}(\Omega):X=0$, $\hat{c}$-quasi surely
(q.s.).$\}$It is seen that $\mathcal{L}^{p}$ and $\mathcal{N}^{p}$ are linear
spaces and $\mathcal{N}^{p}=\mathcal{N}$, for each $p>0$. We denote by
$\mathbb{L}^{p}:=\mathcal{L}^{p}/\mathcal{N}$. As usual, we do not make the
distinction between classes and their representatives.
\end{itemize}

Now, we give the following two propositions which can be found in \cite{DHP}.

\begin{proposition}
\label{pr8} For each $\{X_{n} \}_{n=1}^{\infty}$ in $C_{b}(\Omega)$ such that
$X_{n}\downarrow0$ \ on $\Omega$, we have $\mathbb{\hat{E}} [X_{n}%
]\downarrow0$.
\end{proposition}

\begin{proposition}
\label{pr9} We have

\begin{enumerate}
\item For each $p\geq1$, $\mathbb{L}^{p}$ is a Banach space under the norm
$\left \Vert X\right \Vert _{p}:=\left(  \mathbb{\hat{E}}[|X|^{p}]\right)
^{\frac{1}{p}}$.

\item $\mathbb{L}_{\ast}^{p}$ is the completion of $B_{b}(\Omega)$ under the
Banach norm $\mathbb{\hat{E}}[|X|^{p}]^{1/p}$.

\item $\mathbb{L}_{G}^{p}$ is the completion of $C_{b}(\Omega)$.
\end{enumerate}
\end{proposition}

The following Proposition is obvious.

\begin{proposition}
We have

\begin{enumerate}
\item $\mathbb{L}_{\ast}^{p}\subset \mathbb{L}^{p}\subset \mathbb{L}_{\ast}%
^{q}\subset \mathbb{L}^{q}$, $\mathbb{L}_{q.c.}^{p}\subset \mathbb{L}_{q.c.}%
^{p}$, $0<p\leq q\leq \infty$;

\item $\left \Vert X\right \Vert _{p}\uparrow \left \Vert X\right \Vert _{\infty}$,
for each $X\in \mathbb{L}^{\infty}$;

\item $p,q>1$, $\frac{1}{p}+\frac{1}{q}=1$. Then $X\in \mathbb{L}^{p}$ and
$Y\in \mathbb{L}^{q}$ implies
\[
XY\in \mathbb{L}^{1}\text{ and }\mathbb{E}[|XY|]\leq \left(  \mathbb{E}%
[|X|^{p}]\right)  ^{\frac{1}{p}}\left(  \mathbb{E}[|Y|^{q}]\right)  ^{\frac
{1}{q}}%
\]
\newline Moreover $X\in \mathbb{L}_{c}^{p}$ and $Y\in \mathbb{L}_{c}^{q}$
implies $XY\in \mathbb{L}_{c}^{1};$


\end{enumerate}
\end{proposition}

\begin{proposition}
For a given $p\in(0,+\infty]$, let $\{X_{n}\}_{n=1}^{\infty}$ be a sequence in
$\mathbb{L}^{p}$ which converges to $X$ in $\mathbb{L}^{p}$. Then there exists
a subsequence $(X_{n_{k}})$ which converges to $X$ quasi-surely in the sense
that it converges to $X$ outside a polar set.
\end{proposition}

We also have

\begin{proposition}
\label{Prop5}For each $p>0$,%
\[
\mathbb{L}_{\ast}^{p}=\{X\in \mathbb{L}^{p}:\lim_{n\rightarrow \infty}%
\mathbb{E}[|X|^{p}\mathbf{1}_{\{|X|>n\}}]=0\}.
\]

\end{proposition}

We introduce the following properties. They are important in this paper:

\begin{proposition}
\label{p1} For each $0\leq t<T$, $\xi \in \mathbb{L}^{2}(\Omega_{t})$, we have%
\[
\mathbb{\hat{E}}[\xi(B_{T}-B_{t})]=0.
\]

\end{proposition}

\begin{proof}
Let $P\in \mathcal{P}$ be given. If $\xi \in C_{b}(\Omega_{t})$, then we have
\[
0=-\mathbb{\hat{E}}[-\xi(B_{T}-B_{t})]\leq E_{P}[\xi(B_{T}-B_{t}%
)]\leq \mathbb{\hat{E}}[\xi(B_{T}-B_{t})]=0.
\]
In the case when $\xi \in \mathbb{L}^{2}(\Omega_{t})$, we have $E_{P}[|\xi
|^{2}]\leq \mathbb{\hat{E}}[|\xi|^{2}]<\infty$. Since it is known that
$C_{b}(\Omega_{t})$ is dense in $L_{P}^{2}(\Omega_{t})$, we then can choose a
sequence $\{ \xi_{n}\}_{n=1}^{\infty}$ in $C_{b}(\Omega_{t})$ such that
$E_{P}[|\xi-\xi_{n}|^{2}]\rightarrow0$. Thus
\[
E_{P}[\xi(B_{T}-B_{t})]=\lim_{n\rightarrow \infty}E_{P}[\xi_{n}(B_{T}%
-B_{t})]=0.
\]
The proof is complete.
\end{proof}

\begin{proposition}
\label{p2} For each $0\leq t\leq T$, $\xi \in B_{b}({\Omega_{t}})$, we have
\begin{equation}
\hat{\mathbb{E}}[\xi^{2}(B_{T}-B_{t})^{2}-\overline{\sigma}^{2}\xi
^{2}(T-t)]\leq0. \label{PropP1}%
\end{equation}

\end{proposition}

\begin{proof}
If $\xi \in C_{b}(\Omega_{t})$, then by [Peng], we have the following It\^{o}'s
formula:
\[
\xi^{2}[(B_{T}-B_{t})^{2}-(\langle B_{T}\rangle-\langle B_{t}\rangle
)]=2\int_{t}^{T}\xi^{2}B_{s}dB_{s}.\text{ }%
\]
It follows that $\mathbb{\hat{E}[}\xi^{2}(B_{T}-B_{t})^{2}-\xi^{2}(\langle
B_{T}\rangle-\langle B_{t}\rangle)]=0$. On the other hand, we have $\langle
B_{T}\rangle-\langle B_{t}\rangle \leq \bar{\sigma}^{2}(T-t)$, quasi surely.
Thus (\ref{PropP1}) holds for $\xi \in C_{b}(\Omega_{t})$. It follows that, for
each fixed $P\in \mathcal{P}$, we have
\begin{equation}
E_{P}\mathbb{[}\xi^{2}(B_{T}-B_{t})^{2}-\xi^{2}(\langle B_{T}\rangle-\langle
B_{t}\rangle)]\leq0.\label{PropP1-1}%
\end{equation}
In the case when $\xi \in B_{b}(\Omega_{t})$, we can find a sequence $\{ \xi
_{n}\}_{n=1}^{\infty}$ in $C_{b}(\Omega_{t})$, such that $\xi_{n}%
\rightarrow \xi$ in $L^{p}(\Omega,\mathcal{F}_{t},P)$, for some $p>2$. Thus we
have
\[
E_{P}\mathbb{[}\xi_{n}^{2}(B_{T}-B_{t})^{2}-\xi_{n}^{2}(\langle B_{T}%
\rangle-\langle B_{t}\rangle)]\leq0,
\]
and then, by letting $n\rightarrow \infty$, obtain (\ref{PropP1-1}) for $\xi \in
B_{b}(\Omega_{t})$. Thus (\ref{PropP1}) follows immediately for $\xi \in
B_{b}(\Omega_{t})$.
\end{proof}

\section{A generalized Ito's Integral}

For notational simplification, in the rest of the paper we only discuss
$1$-dimensional Brownian motion, i.e., $d=1$. But all the results can be
generalized to multi-dimensional situation. We refer to
\cite{Peng2006b,Peng2007} for the corresponding techniques. For $p\geq1$ and
$T\in \mathbb{R}_{+}$ be fixed, we first consider the following simple type of
processes:
\begin{align*}
M_{b,0}(0,T) &  =\{ \eta:\eta_{t}(\omega)=\sum_{j=0}^{N-1}\xi_{j}%
(\omega)\mathbf{I}_{[t_{j},t_{j+1})}(t),\\
&  \forall N>0,0=t_{0}<\cdots<t_{N}=T,\xi_{j}(\omega)\in B_{b}(\Omega_{t_{j}%
}),j=0,\cdots,N-1\}.
\end{align*}

\begin{definition}
For an $\eta \in M_{b,0}(0,T)$ with $\eta_{t}=\sum_{j=0}^{N-1}\xi_{j}%
(\omega)\mathbf{I}_{[t_{j},t_{j}+1)}(t)$, the related Bochner integral is
\[
\int_{0}^{T}\eta_{t}(\omega)dt=\sum_{j=0}^{N-1}\xi_{j}(\omega)(t_{j+1}%
-t_{j}).
\]

\end{definition}

For each $\eta \in M_{b,0}(0,T)$ we set
\[
\hat{\mathbb{E}}_{T}[\eta]:=\frac{1}{T}\hat{\mathbb{E}}[\int_{0}^{T}\eta
_{t}dt]=\frac{1}{T}\hat{\mathbb{E}}[\sum_{j=0}^{N-1}\xi_{j}(\omega
)(t_{j+1}-t_{j})].
\]

We can introduce a natural norm $||\eta||_{M^{p}(0,T)}=\{ \hat{\mathbb{E}%
}[\int_{0}^{T}|\eta_{t}|^{p}dt]\}^{1/p}$. Under this norm, $M_{b,0}(0,T)$ can
be continuously extended to a Banach space.

\begin{definition}
For each $p\geq1$, we denote by $M_{\ast}^{p}(0,T)$ the completion of
$M_{b,0}(0,T)$ under the norm
\[
||\eta||_{M^{p}(0,T)}=\{ \hat{\mathbb{E}}[\int_{0}^{T}|\eta_{t}|^{p}%
dt]\}^{1/p}.
\]

\end{definition}

We have $M_{\ast}^{p}(0,T)\supset M_{\ast}^{q}(0,T)$, for $p\leq q$. The
following process
\[
\eta_{t}(\omega)=\sum_{j=0}^{N-1}\xi_{j}(\omega)\mathbf{I}_{[t_{j},t_{j+1}%
)}(t),\  \xi_{j}\in \mathbb{L}_{\ast}^{p}(\Omega_{t_{j}}),\ j=1,\cdots,N
\]
is also in $M_{\ast}^{p}(0,T)$.

\begin{definition}
For each $\eta \in M_{b,0}(0,T)$ with the form
\[
\eta_{t}(\omega)=\sum_{j=0}^{N-1}\xi_{j}(\omega)\mathbf{I}_{[t_{j},t_{j+1}%
)}(t),
\]
we define It\^{o}'s integral
\[
I(\eta)=\int_{0}^{T}\eta_{s}dB_{s}:=\sum_{j=0}^{N-1}\xi_{j}(B_{t_{j+1}%
}-B_{t_{j}})\mathbf{.}%
\]

\end{definition}

\begin{lemma}
{ { { \label{bdd}The mapping $I:M_{\ast}^{2}(0,T)\rightarrow \mathbb{L}%
^{2}(\Omega_{T})$ is a linear continuous mapping and thus can be continuously
extended to $I:M_{\ast}^{2}(0,T)\rightarrow \mathbb{L}^{2}(\Omega_{T})$. We
have
\begin{align}
\mathbb{\hat{E}}[\int_{0}^{T}\eta_{s}dB_{s}]  &  =0,\  \  \label{e1}\\
\mathbb{\hat{E}}[(\int_{0}^{T}\eta_{s}dB_{s})^{2}]  &  \leq \overline{\sigma
}^{2}\hat{\mathbb{E}}[\int_{0}^{T}\eta_{t}^{2}dt]. \label{e2}%
\end{align}
} } }
\end{lemma}

\begin{proof}
We only need to prove (\ref{e1}) and (\ref{e2}). From Proposition \ref{p1},
for each $j$,
\[
\mathbb{\hat{E}}\mathbf{[}\xi_{j}(B_{t_{j+1}}-B_{t_{j}})]=\mathbb{\hat{E}%
}\mathbf{[-}\xi_{j}(B_{t_{j+1}}-B_{t_{j}})]=0.
\]
Thus we have%
\begin{align*}
\mathbb{\hat{E}}[\int_{0}^{T}\eta_{s}dB_{s}]  &  =\mathbb{\hat{E}[}\int
_{0}^{t_{N-1}}\eta_{s}dB_{s}+\xi_{N-1}(B_{t_{N}}-B_{t_{N-1}})]\\
&  =\mathbb{\hat{E}[}\int_{0}^{t_{N-1}}\eta_{s}dB_{s}]=\cdots=\hat{\mathbb{E}%
}[\xi_{0}(B_{t_{1}}-B_{t_{0}})]=0.
\end{align*}

We now prove (\ref{e2}), we first apply Proposition \ref{p1} to derive
\begin{align*}
\hat{\mathbb{E}}[(\int_{0}^{T}\eta_{t}dB_{t})^{2}]  &  =\hat{\mathbb{E}%
}\mathbb{[}\left(  \int_{0}^{t_{N-1}}\eta_{t}dB_{t}+\xi_{N-1}(B_{t_{N}%
}-B_{t_{N-1}})\right)  ^{2}]\\
&  =\hat{\mathbb{E}}\mathbb{[}\left(  \int_{0}^{t_{N-1}}\eta(t)dB_{t}\right)
^{2}+\xi_{N-1}^{2}(B_{t_{N}}-B_{t_{N-1}})^{2}\\
&  +2\left(  \int_{0}^{t_{N-1}}\eta_{t}dB_{t}\right)  \xi_{N-1}(B_{t_{N}%
}-B_{t_{N-1}})]\\
&  =\hat{\mathbb{E}}\mathbb{[}\left(  \int_{0}^{t_{N-1}}\eta_{t}dB_{t}\right)
^{2}+\xi_{N-1}^{2}(B_{t_{N}}-B_{t_{N-1}})^{2}]\\
&  =\cdots=\hat{\mathbb{E}}[\sum_{i=0}^{N-1}\xi_{i}^{2}(B_{t_{i+1}}-B_{t_{i}%
})^{2}].
\end{align*}

Then by Proposition \ref{p2}, we have
\[
\hat{\mathbb{E}}[\xi_{j}^{2}(B_{t_{j+1}}-B_{t_{j}})^{2}-\overline{\sigma}%
\xi_{j}^{2}(t_{j+1}-t_{j})]\leq0.
\]
Thus
\begin{align*}
&  \hat{\mathbb{E}}[(\int_{0}^{T}\eta_{t}dB_{t})^{2}]=\hat{\mathbb{E}}%
[\sum_{i=0}^{N-1}\xi_{i}^{2}(B_{t_{N}}-B_{t_{N-1}})^{2}]\\
\leq &  \hat{\mathbb{E}}[\sum_{i=0}^{N-1}\xi_{i}^{2}[(B_{t_{N}}-B_{t_{N-1}%
})^{2}-\overline{\sigma}^{2}(t_{i+1}-t_{i})]]+\hat{\mathbb{E}}[\sum
_{i=0}^{N-1}\overline{\sigma}^{2}\xi_{i}^{2}(t_{i+1}-t_{i})]\\
\leq &  \sum_{i=0}^{N-1}\hat{\mathbb{E}}[\xi_{i}^{2}(B_{t_{j+1}}-B_{t_{j}%
})^{2}-\overline{\sigma}^{2}\xi_{i}^{2}(t_{j+1}-t_{j})]+\hat{\mathbb{E}}%
[\sum_{i=0}^{N-1}\overline{\sigma}^{2}\xi_{i}^{2}(t_{i+1}-t_{i})]\\
\leq &  \hat{\mathbb{E}}[\sum_{i=0}^{N-1}\overline{\sigma}^{2}\xi_{i}%
^{2}(t_{i+1}-t_{i})]=\bar{\sigma}^{2}\hat{\mathbb{E}}[\int_{0}^{T}\eta_{t}%
^{2}dt].
\end{align*}

\end{proof}

The following Proposition can be verified directly by the definition of
It\^{o}'s integral with respect to $G$-Brownian motion.

\begin{proposition}
\label{p5} Let $\eta,\theta \in M_{\ast}^{2}(0,T)$, and let $0\leq s\leq r\leq
t\leq T$. Then we have

1. $\int_{s}^{t}\eta_{u}d B_{u}=\int_{s}^{r} \eta_{u}d B_{u}+\int_{r}^{t}%
\eta_{u}d B_{u}, q.s.,$

2. $\int_{s}^{t}(\alpha \eta_{u}+\theta_{u})dB_{u}=\alpha \int_{s}^{t}\eta
_{u}dB_{u}+\int_{s}^{t}\theta_{u}dB_{u}$, where $\alpha \in B_{b}(\Omega_{s})$.
\end{proposition}

\begin{proposition}
\label{p3} For each $\eta \in M_{\ast}^{2}(0,T)$, we have
\begin{equation}
\hat{\mathbb{E}}[\sup_{0\leq t\leq T}|\int_{0}^{t}\eta_{s}dB_{s}|^{2}%
]\leq2\overline{\sigma}^{2}\hat{\mathbb{E}}[\int_{0}^{T}\eta_{s}%
^{2}ds].\label{6.0}%
\end{equation}

\end{proposition}

\begin{proof}
Since for each $\alpha \in B_{b}(\Omega_{t})$, we have%
\[
\mathbb{\hat{E}}[\alpha \int_{t}^{T}\eta_{s}dB_{s}]=0\text{,}%
\]
thus, for each fixed $P\in \mathcal{P}$, the process $\int_{0}^{\cdot}\eta
_{s}dB_{s}$ is a $P$-martingale. it follows from the classical Doob's
martingale inequality that
\[
E_{P}[\sup_{0\leq t\leq T}|\int_{0}^{t}\eta_{s}dB_{s}|^{2}]\leq2E_{P}%
[|\int_{0}^{T}\eta_{s}dB_{s}|^{2}]\leq2\overline{\sigma}^{2}E_{P}[\int_{0}%
^{T}\eta_{s}^{2}ds]\leq2\overline{\sigma}^{2}\hat{\mathbb{E}}[\int_{0}^{T}%
\eta_{s}^{2}ds].
\]
Thus (\ref{6.0}) holds.
\end{proof}

\begin{proposition}
\label{p7} For any $\eta \in M_{\ast}^{2}(0,T)$ and $0\leq t\leq T$, $\int
_{0}^{t}\eta_{s}dB_{s}$ is continuous in $t$ quasi-surely.
\end{proposition}

\begin{proof}
The claim is true for $\eta \in M_{b,0}(0,T)$ since $(B_{t})_{t\geq0}$ is
quasi-surely continuous. In the case when $\eta \in M_{\ast}^{2}(0,T)$, there
exists $\eta^{n}\in M_{b,0}(0,T)$, such that $\hat{\mathbb{E}}[\int_{0}%
^{T}(\eta_{s}-\eta_{s}^{n})^{2}ds]\rightarrow0$. By Proposition \ref{p3}, we
have
\[
\hat{\mathbb{E}}[\sup_{0\leq t\leq T}|\int_{0}^{t}(\eta_{s}-\eta_{s}%
^{n})dB_{s}|^{2}]\leq2\overline{\sigma}^{2}\hat{\mathbb{E}}[\int_{0}^{T}%
(\eta_{s}-\eta_{s}^{n})^{2}ds]\rightarrow0.
\]
This implies that, quasi-surely, the sequence of processes $\int_{0}^{\cdot
}\eta_{s}^{n}dB_{s}$ uniformly converges to $\int_{0}^{\cdot}\eta_{s}dB_{s}$
on $[0,T]$. Thus $\int_{0}^{t}\eta_{s}dB_{s}$ is continuous in $t$ quasi-surely.
\end{proof}

We have also the following

\begin{proposition}
Let $X\in M_{\ast}^{p+\varepsilon}(\Omega \times \lbrack0,T])$ with $p\geq1$ and
$\varepsilon>0$. Then we have%
\[
\mathbb{\hat{E}}\int_{0}^{T}|X_{t}|^{p}I_{\{|X_{t}|>n\}}dt\rightarrow0\text{,
as }n\rightarrow \infty.
\]

\end{proposition}

\begin{proof}
It is clear that $X\in M_{\ast}^{p}(0,T)$. Moreover we have
\begin{align*}
\mathbb{\hat{E}}\int_{0}^{T}|X_{t}|^{p}I_{\{|X_{t}|>n\}}dt  &  \leq \left[
\mathbb{\hat{E}}\int_{0}^{T}|X_{t}|^{p+\varepsilon}dt\right]  ^{\frac
{p}{p+\varepsilon}}\cdot \left[  \mathbb{\hat{E}}\int_{0}^{T}I_{\{|X_{t}%
|>n\}}dt\right]  ^{\frac{\varepsilon}{p+\varepsilon}}\\
&  \leq \left[  \mathbb{\hat{E}}\int_{0}^{T}|X_{t}|^{p+\varepsilon}dt\right]
^{\frac{p}{p+\varepsilon}}\left[  \mathbb{\hat{E}}\int_{0}^{T}n^{-p}%
|X_{t}|^{p}dt\right]  ^{\frac{\varepsilon}{p+\varepsilon}}\rightarrow0,
\end{align*}
as $n\rightarrow \infty$.
\end{proof}

\begin{proposition}
\label{Gt13} For each $p\geq1$ and $X\in M_{\ast}^{p}(0,T)$ we have
\begin{equation}
\lim_{n\rightarrow \infty}\mathbb{\hat{E}}\int_{0}^{T}[|X_{t}|^{p}%
I_{\{|X_{t}|>n\}}]dt=0. \label{Jp}%
\end{equation}

\end{proposition}

\begin{proof}
For each $X\in M_{\ast}^{p}(0,T)$, we can find a sequence $\left \{
Y^{(n)}\right \}  _{n=1}^{\infty}$ in $M_{b,0}(0,T)$ such that $\mathbb{\hat
{E}}\int_{0}^{T}[|X_{t}-Y_{t}^{(n)}|^{p}]dt\rightarrow0$. Let $y_{n}%
=\sup_{\omega \in \Omega,t\in \lbrack0,T]}|Y_{t}^{(n)}(\omega)|$ and
$X^{(n)}=(X\wedge y_{n})\vee(-y_{n})$. Since $|X-X^{(n)}|\leq|X-Y^{(n)}|$, we
have $\mathbb{\hat{E}}\int_{0}^{T}[|X_{t}-X_{t}^{(n)}|^{p}]dt\rightarrow0$.
This also implies that for any sequence $\{ \alpha_{n}\}$ tending to $\infty$,
\[
\lim_{n\rightarrow \infty}\mathbb{\hat{E}}\int_{0}^{T}[|X_{t}-(X_{t}%
\wedge \alpha_{n})\vee(-\alpha_{n})|^{p}]=0.\newline%
\]
Now we have for all $n\in \mathbb{N}$,
\begin{align*}
&  \mathbb{\hat{E}}\int_{0}^{T}|X_{t}|^{p}I_{\{|X_{t}|>n\}}dt=\mathbb{\hat{E}%
}\int_{0}^{T}(|X_{t}|-n+n)^{p}I_{\{|X_{t}|>n\}}dt\\
&  \leq(1\vee2^{p-1})\left(  \mathbb{\hat{E}}\int_{0}^{T}[(|X_{t}%
|-n)^{p}I_{\{|X_{t}|>n\}}]dt+n^{p}\mathbb{\hat{E}}\int_{0}^{T}I_{\{|X_{t}%
|>n\}}dt\right)  .
\end{align*}
The first term of the right hand side tends to $0$ since%
\[
\mathbb{\hat{E}}[\int_{0}^{T}(|X_{t}|-n)^{p}I_{\{|X_{t}|>n\}}dt]=\mathbb{\hat
{E}}[\int_{0}^{T}|X_{t}-(X_{t}\wedge{n})\vee{(-n)}|^{p}dt]\rightarrow0.
\]
For the second term, since%
\[
\frac{n^{p}}{2^{p}}I_{\{|X_{t}|>n\}}\leq(|X_{t}|-\frac{n}{2})^{p}%
I_{\{|X_{t}|>n\}}\leq(|X_{t}|-\frac{n}{2})^{p}I_{\{|X_{t}|>\frac{n}{2}\}},
\]
thus we have
\[
\frac{n^{p}}{2^{p}}\mathbb{\hat{E}}\int_{0}^{T}I_{\{|X_{t}|>n\}}%
dt=\mathbb{\hat{E}}\int_{0}^{T}(|X|-\frac{n}{2})^{p}I_{\{|X|>\frac{n}{2}%
\}}dt\rightarrow0.
\]
Consequently (\ref{Jp}) holds true for $X\in M_{\ast}^{p}(0,T)$.
\end{proof}

\begin{corollary}
\label{p8} For each $\eta \in M_{\ast}^{2}(0,T)$, let $\eta_{s}^{n}%
=(-n)\vee(\eta_{s}\wedge n)$, then we have $\int_{0}^{t}\eta_{s}^{n}%
dB_{s}\rightarrow \int_{0}^{t}\eta_{s}dB_{s}$ in $M_{\ast}^{2}(0,T)$ for each
$t\leq T$.
\end{corollary}

\begin{proposition}
\label{Gt14} Let $X\in M_{\ast}^{p}(0,T)$. Then for each $\varepsilon>0$,
there exists $\delta>0$ such that for all $\eta \in M_{b,0}(0,T)$ satisfying
$\mathbb{\hat{E}}\int_{0}^{T}|\eta_{t}|dt\leq \delta$ and $|\eta_{t}%
(\omega)|\leq1$, we have $\mathbb{\hat{E}}\int_{0}^{T}[|X_{t}|^{p}|\eta
_{t}|]dt\leq \varepsilon$.
\end{proposition}

\begin{proof}
For each $\varepsilon>0$, by Proposition \ref{Gt13}, there exists $N>0$ such
that $\mathbb{\hat{E}}[\int_{0}^{T}|X|^{p}I_{\{|X|>N\}}]\leq \frac{\varepsilon
}{2}$. Take $\delta=\frac{\varepsilon}{2N^{p}T}$. Then we have%
\begin{align*}
\mathbb{\hat{E}}\int_{0}^{T}[|X_{t}|^{p}|\eta_{t}|]dt  &  \leq \mathbb{\hat{E}%
}\int_{0}^{T}|X_{t}|^{p}|\eta_{t}|I_{\{|X_{t}|>N\}}dt+\mathbb{\hat{E}}\int
_{0}^{T}|X_{t}|^{p}|\eta_{t}|I_{\{|X_{t}|\leq N\}}dt\\
&  \leq \mathbb{\hat{E}}\int_{0}^{T}|X_{t}|^{p}I_{\{|X_{t}|>N\}}dt+N^{p}%
\mathbb{\hat{E}}\int_{0}^{T}|\eta_{t}|dt\leq \varepsilon \text{.}%
\end{align*}

\end{proof}

\begin{lemma}
\label{Lemm3.12}If $p\geq1$, $X,\eta \in M_{\ast}^{p}(0,T)$ such that $\eta$ is
bounded, then $X\eta \in M_{\ast}^{p}(0,T)$.
\end{lemma}

\begin{proof}
Let $c>0$ be such that $|\eta_{t}(\omega)|\leq c$, for $\omega \in \Omega$,
$t\in \lbrack0,T]$. Then we have
\begin{align*}
\mathbb{\hat{E}}\int_{0}^{T}|\eta_{t}X_{t}|^{p}I_{\{|\eta_{t}X_{t}|>N\}}dt &
\leq c^{p}\mathbb{\hat{E}}\int_{0}^{T}|X_{t}|^{p}I_{\{|cX_{t}|>N\}}dt\\
&  \leq c^{p}\mathbb{\hat{E}}\int_{0}^{T}|X_{t}|^{p}I_{\{|X_{t}|>\frac{N}%
{c}\}}dt\rightarrow0\text{, as }N\rightarrow \infty.
\end{align*}
It follows that the bounded $\{Y\}_{n=1}^{\infty}:=\{(-n)\vee(n\wedge(\eta
X))\}_{n=1}^{\infty}$ is a Cauchy sequence in $M_{\ast}^{p}(0,T)$.
\end{proof}

\begin{remark}
It is easy to prove that if $\eta \in M_{\ast}^{2}(0,T)$, then $\int_{0}%
^{\cdot}\eta_{s}dB_{s}\in M_{\ast}^{2}(0,T)$.
\end{remark}

\section{Ito's integral with stopping times}

In this section we study It\^{o}'s integral on a interval $[0,\tau]$, where
$\tau$ is a stopping time. Reader can see that, thanks to Propositions
\ref{Gt13}, \ref{Gt14} and Lemma \ref{Lemm3.12}, the techniques used in this
section is very similar to the clasical situation.

\begin{definition}
A stopping time $\tau$ relative to the filtration $(\mathcal{F}_{t})$ is a map
on $\Omega$ with values in $[0,T]$, such that for every $t$,
\[
\{ \tau \leq t\} \in \mathcal{F}_{t}.
\]

\end{definition}

\begin{lemma}
\label{t1} For each stopping time $\tau$, we have $\mathbf{I}_{[0,\tau]}%
(\cdot)X\in M_{\ast}^{p}(0,T)$, for each $X\in M_{\ast}^{p}(0,T)$.
\end{lemma}

\begin{proof}
For the given stopping time $\tau$, let
\[
\tau_{n}=\sum_{k=0}^{2^{n}-1}\frac{k}{2^{n}}\mathbf{I}_{[\frac{kT}{2^{n}}%
\leq \tau<\frac{(k+1)T}{2^{n}})}+T\mathbf{I}_{[\tau \geq T]}.
\]
Then we have $2^{-n}\geq \tau_{n}-\tau \geq0$. It is clear that, for $m\geq n$,
\begin{align*}
\mathbb{\hat{E}}\int_{0}^{T}|I_{[0,\tau_{n}]}(t)-I_{[0,\tau_{m}]}(t)|dt  &
\leq \mathbb{\hat{E}}\int_{0}^{T}|I_{[0,\tau_{n}]}(t)-I_{[0,\tau]}(t)|dt\\
&  =\mathbb{\hat{E}}[\tau_{n}-\tau]\leq2^{-n}T\text{.}%
\end{align*}

It follows from Proposition \ref{Gt14} that For each $\tau_{n}$, it is easy to
check that $\{I_{[0,\tau_{n}]}X\}_{n=1}^{\infty}$ is a Cauchy sequence in
$M_{\ast}^{p}(0,T)$. Thus $I_{[0,\tau]}X\in M_{\ast}^{p}(0,T)$.
\end{proof}

\begin{lemma}
For each $\eta \in M_{\ast}^{p}(0,T)$ and $\tau$ be a stopping time, then
\begin{equation}
\int_{0}^{t\wedge \tau}\eta_{s}dB_{s}=\int_{0}^{t}\mathbf{I}_{[0,\tau]}%
(s)\eta_{s}dB_{s},\  \text{quasi-surely.} \label{e3-1}%
\end{equation}

\end{lemma}

\begin{proof}
For each $n\in \mathbb{N}$, let
\[
\tau_{n}:=\sum_{k=1}^{[t\cdot2^{n}]}\frac{k}{2^{n}}\mathbf{I}_{[\frac
{(k-1)t}{2^{n}}\leq \tau<\frac{kt}{2^{n}})}+t\mathbf{I}_{[\tau \geq t]}%
=\sum_{k=1}^{2^{n}}\mathbf{I}_{A_{n}^{k}}t_{n}^{k}.
\]
where $t_{n}^{k}=k2^{-n}t$, $A_{n}^{k}=[t_{n}^{k-1}<t\wedge \tau \leq t_{n}%
^{k}]$, for $k<2^{n}$, and $A_{n}^{2^{n}}=[\tau \geq t]$. $\{ \tau_{n}%
\}_{n=1}^{\infty}$ is a decreasing sequence of stopping times which converges
q.s. to $t\wedge \tau$.

We first prove that
\begin{equation}
\int_{\tau_{n}}^{t}\eta_{s}dB_{s}=\int_{0}^{t}\mathbf{I}_{[\tau_{n},t]}%
(s)\eta_{s}dB_{s}. \label{e3}%
\end{equation}

But by Proposition \ref{p5} we have
\begin{align*}
\int_{\tau_{n}}^{t}\eta_{s}dB_{s}  &  =\int_{\sum_{k=1}^{2^{n}}\mathbf{I}%
_{A_{n}^{k}}t_{n}^{k}}^{t}\eta_{s}dB_{s}=\sum_{k=1}^{2^{n}}\mathbf{I}%
_{A_{n}^{k}}\int_{t_{n}^{k}}^{t}\eta_{s}dB_{s}\\
&  =\sum_{k=1}^{2^{n}}\int_{t_{n}^{k}}^{t}\mathbf{I}_{A_{n}^{k}}\eta_{s}%
dB_{s}\\
&  =\int_{0}^{t}\sum_{k=1}^{2^{n}}\mathbf{I}_{[t_{n}^{k},t]}(s)\mathbf{I}%
_{A_{n}^{k}}\eta_{s}dB_{s},
\end{align*}
from which (\ref{e3}) follows. We thus have%
\[
\int_{0}^{\tau_{n}}\eta_{s}dB_{s}=\int_{0}^{t}\mathbf{I}_{[0,\tau_{n}]}%
(s)\eta_{s}dB_{s}.
\]
Observe that $0\leq \tau_{n}-\tau_{m}\leq \tau_{n}-t\wedge \tau \leq2^{-n}t$, for
$n\leq m$, this with Proposition \ref{Gt14} it follows that $\mathbf{I}%
_{[0,\tau_{n}]}\eta$ converges in $M_{\ast}^{2}(0,T)$ to $\mathbf{I}%
_{[0,\tau \wedge t]}\eta$ and thus $\mathbf{I}_{[0,\tau \wedge t]}\eta \in
M_{\ast}^{2}(0,T)$. Consequently,
\[
\lim_{n\rightarrow \infty}\int_{0}^{\tau_{n}}\eta_{s}dB_{s}=\int_{0}%
^{t\wedge \tau}\eta_{s}dB_{s},\  \  \text{quasi-surely}%
\]
and (\ref{e3-1}) holds as well.
\end{proof}

\begin{definition}
\label{Def4.4}Let $p>0$ be fixed. A stochastic process $(\eta_{t})_{t\geq0}$
with $\eta_{t}\in \mathbb{L}^{0}(\Omega_{t})$ is said to be in $M_{\omega}%
^{p}(0,T)$ if there exists a sequence of increasing stopping times
$\{ \sigma_{m}\}_{m=1}^{\infty}$, with $\sigma_{m}\uparrow T$, quasi-surely,
such that $\eta \mathbf{I}_{[0,\sigma_{m}]}\in M_{\ast}^{p}(0,T)$ and
\[
\inf_{P\in \mathcal{P}}P(\int_{0}^{T}|\eta_{s}|^{p}ds<\infty)=1.\  \  \text{ }%
\]

\end{definition}

\begin{remark}
In the rest of this paper the notation $\{ \sigma_{m}\}_{m=1}^{\infty}$ is used
to denote the sequence of the corresponding process $\eta \in M_{\omega}%
^{p}(0,T)$. Let $\{ \tau_{m}\}_{m=1}^{\infty}$ be another sequence of
increasing stopping times with $\tau_{m}\uparrow T$, quasi-surely. Then it is
easy to check that $\eta \mathbf{I}_{[0,\sigma_{m}\wedge \tau_{m}]}\in M_{\ast
}^{p}(0,T)$ and $\sigma_{m}\wedge \tau_{m}\uparrow T$, quasi-surely. Thus we
can as well use $\sigma_{m}\wedge \tau_{m}$ in the place of $\sigma_{m}$. For
example, when we consider two processes $\eta$, $\bar{\eta}\in M_{\omega
}^{\ast}(0,T)$ with two sequences of stopping times $\{ \sigma_{m}%
\}_{m=1}^{\infty}$ and $\{ \bar{\sigma}_{m}\}_{m=1}^{\infty}$, we may only use
one sequence $\{ \sigma_{m}\wedge \bar{\sigma}_{m}\}_{m=1}^{\infty}$ for both
$\eta$ and $\bar{\eta}$.
\end{remark}

\begin{lemma}
\label{Lemm4.6}Let $\eta \in M_{\omega}^{1}(0,T)$ be given and let
\[
\tau_{n}=\inf \{t\geq0,\int_{0}^{t}|\eta_{s}|ds>n\} \wedge \sigma_{n}.
\]
Then $\eta \mathbf{I}_{[0,\tau_{n}]}\in M_{\ast}^{1}(0,T)$ and $\int_{0}%
^{t}\mathbf{I}_{[0,\tau_{n}]}(s)\eta_{s}ds$ and $\int_{0}^{t}\mathbf{I}%
_{[0,\tau_{n}]}(s)\eta_{s}d\left \langle B\right \rangle _{s}$ are well-defined
processes which are continuous on $[0,T]$ quasi-surely.
\end{lemma}

The proof is similar to that of the following proposition.

\begin{proposition}
\label{Prop4.7}let $\tau_{n}=\inf \{t\geq0,\int_{0}^{t}|\eta_{s}|^{2}%
ds>n\} \wedge \sigma_{n}$ and $\Omega_{n}=\{ \tau_{n}=T\}$. Then $\eta
\mathbf{I}_{[0,\tau_{n}]}\in M_{\ast}^{2}(0,T)$ and the stochastic process
$(\int_{0}^{t}\eta_{s}dB_{s})_{t\in \lbrack0,T]}$ is a well-defined
quasi-surely continuous process defined on $\Omega$.
\end{proposition}

\begin{proof}
Since, for each $n=1,2,\cdots$, $\eta \mathbf{I}_{[0,\tau_{n}]}\in M_{\ast}%
^{2}(0,T)$, so the It\^{o}'s integral $\int_{0}^{t}\mathbf{I}_{[0,\tau_{n}%
]}(s)\eta_{s}dB_{s}$ is well-defined. On the other hand, on the subset
$\Omega_{n}=\{ \tau_{n}=T\}$ and for each $m>n\,$, we have $\tau_{m}=\tau
_{n}=T$. Thus
\[
\lim_{m\rightarrow \infty}\mathbf{I}_{\Omega_{n}}\int_{0}^{\tau_{m}\wedge
t}\eta_{s}dB_{s}=\mathbf{I}_{\Omega_{n}}\int_{0}^{\tau_{n}\wedge t}\eta
_{s}dB_{s}=\mathbf{I}_{\Omega_{n}}\int_{0}^{t}\mathbf{I}_{[0,\tau_{n}]}%
(s)\eta_{s}dB_{s},\ t\in \lbrack0,T].
\]%
\[
\lim_{m\rightarrow \infty}\mathbf{I}_{\Omega_{n}}\int_{0}^{\tau_{m}\wedge
t}\eta_{s}dB_{s}=\mathbf{I}_{\Omega_{n}}\int_{0}^{t}\eta_{s}dB_{s}%
,\  \ t\in \lbrack0,T]
\]
Thus on $\Omega_{n}$ the process $(\int_{0}^{t}\eta_{s}dB_{s})_{t\in
\lbrack0,T]}$ is a well-defined process which is continuous in $t$
quasi-surely. Since $\Omega_{n}\uparrow \bar{\Omega}\subset \Omega$, with
$\hat{c}(\bar{\Omega}^{c})=0$. It follows $(\int_{0}^{t}\eta_{s}dB_{s}%
)_{t\in \lbrack0,T]}$ can is a well-defined process which is continuous in $t$
quasi-surely. \ 
\end{proof}

\begin{corollary}
\label{Coro4.8}We assume that $\varphi \in C^{1,2}([0,\infty)\times \mathbb{R})$
and all first and second order derivatives of $\varphi$ with respect to
$(t,x)$ are bounded. Let $\alpha,\eta,\beta \in M_{\omega}^{2}(0,T)$ and
$X_{t}=\int_{0}^{t}\alpha_{s}ds+\int_{0}^{t}\eta_{s}d\langle B\rangle_{s}%
+\int_{0}^{t}\beta_{s}dB_{s}$, $t\in \lbrack0,T]$. Then, for each $\varphi \in
C(\mathbb{R})$ and $\gamma \in M_{\omega}^{p}(0,T)$, $\varphi(X)\gamma \in
M_{\omega}^{p}(0,T)$.
\end{corollary}

The proof is easy since $(X_{t})_{t\geq0}$ is a quasi-surely continuous process.

\section{It\^{o}'s Formula}

\begin{lemma}
\label{lemma1} We assume that $\varphi \in C^{2}(\mathbb{R}^{n})$ and all first
and second order derivatives of $\varphi$ with respect to $x$ are bounded. Let
$X=(X^{1},\cdots,X^{n})$ and
\[
X_{t}^{i}=X_{0}+\int_{0}^{t}\alpha_{s}^{i}ds+\int_{0}^{t}\eta_{s}^{i}d\langle
B\rangle_{s}+\int_{0}^{t}\beta_{s}^{i}dB_{s},\ i=1,\cdots,n.
\]
where $\alpha,\beta,\eta$ are bonded elements in $M_{\ast}^{2}(0,T)$. Then for
each $t\geq0$, we have
\begin{align*}
\varphi(X_{t})-\varphi(X_{0}) &  =\int_{0}^{t}\partial_{x_{i}}\varphi
(X_{u})\beta_{u}^{i}dB_{u}+\int_{0}^{t}\partial_{x_{i}}\varphi(X_{u}%
)\alpha_{u}^{i}du\\
&  +\int_{0}^{t}[\partial_{x_{i}}\varphi(X_{u})\eta_{u}^{i}+\frac{1}%
{2}\partial_{x_{i}x_{j}}^{2}\varphi(X_{u})\beta_{u}^{i}\beta_{u}^{j}]d\langle
B\rangle_{u}.
\end{align*}
Here and in the rest of this paper we use the Einstein convention, i.e., the
above repeated indices of $i$ and $j$ within one term imply the summation from
$1$ to $n$.
\end{lemma}

The proof will be given in the appendix.

\begin{lemma}
\label{lemma5.1} Let $\varphi \in C^{2}(\mathbb{R}^{n})$ and its first and
second derivatives are in $C_{b,Lip}(\mathbb{R}^{n})$. Let $X_{t}^{i}%
=X_{0}^{i}+\int_{0}^{t}\alpha_{s}^{i}ds+\int_{0}^{t}\eta_{s}^{i}d\langle
B\rangle_{s}+\int_{0}^{t}\beta_{s}^{i}dB_{s}$, where $\alpha,\eta$ in
$M_{\ast}^{1}(0,T)$, $\beta \in M_{\ast}^{2}(0,T)$. Then for each $t\geq0$, we
have
\begin{align}
\varphi(X_{t})-\varphi(X_{0}) &  =\int_{0}^{t}\partial_{x_{i}}\varphi
(X_{u})\beta_{u}^{i}dB_{u}+\int_{0}^{t}\partial_{x_{i}}\varphi(X_{u}%
)\alpha_{u}^{i}du\label{eq5.1}\\
&  +\int_{0}^{t}[\partial_{x_{i}}\varphi(X_{u})\eta_{u}^{i}+\frac{1}%
{2}\partial_{x_{i}x_{j}}^{2}\varphi(X_{u})\beta_{u}^{i}\beta_{u}^{j}]d\langle
B\rangle_{u}.\nonumber
\end{align}

\end{lemma}

\begin{proof}
For simplicity, we only state for the case where $n=1$. Let $\alpha^{(k)}$,
$\beta^{(k)}$ and $\eta^{(k)}$ bounded processes such that, as $k\rightarrow
\infty$,
\[
\alpha^{(k)}\rightarrow \alpha,\eta^{(k)}\rightarrow \eta \text{ in }M_{\ast}%
^{1}(0,T)\text{ and }\beta^{(k)}\rightarrow \beta \text{, in }M_{\ast}^{2}(0,T)
\]
and let
\[
X_{t}^{(k)}=X_{0}+\int_{0}^{t}\alpha_{s}^{(k)}ds+\int_{0}^{t}\eta_{s}%
^{(k)}d\langle B\rangle_{s}+\int_{0}^{t}\beta_{s}^{(k)}dB_{s}.
\]
Then we have%
\[
\lim_{k\rightarrow \infty}\mathbb{\hat{E}}[\sup_{0\leq t\leq T}|X_{t}%
-X_{t}^{(k)}|^{2}]=0.
\]
We see that%
\begin{align*}
\mathbb{\hat{E}}\int_{0}^{T}|\partial_{x}\varphi(X_{t}^{(k)})\beta_{t}%
^{(k)}-\partial_{x}\varphi(X_{t})\beta_{t}|^{2}dt &  \leq \mathbb{\hat{E}}%
\int_{0}^{T}|\partial_{x}\varphi(X_{t}^{(k)})\beta_{t}^{(k)}-\partial
_{x}\varphi(X_{t}^{(k)})\beta_{t}|^{2}dt\\
&  +\mathbb{\hat{E}}\int_{0}^{T}|\partial_{x}\varphi(X_{t}^{(k)})\beta
_{t}-\partial_{x}\varphi(X_{t})\beta_{t}|^{2}dt\\
&  \leq C\mathbb{\hat{E}}\int_{0}^{T}|\beta_{t}^{(k)}-\beta_{t}|^{2}dt\\
&  +\mathbb{\hat{E}[}\int_{0}^{T}|\beta_{t}|^{2}|\partial_{x}\varphi
(X_{t}^{(k)})-\partial_{x}\varphi(X_{t})|^{2}dt].
\end{align*}
But we have $\sup_{0\leq t\leq T}|\partial_{x}\varphi(X_{t}^{(k)}%
)-\partial_{x}\varphi(X_{t})|^{2}\leq c$ and
\[
\mathbb{\hat{E}[}\int_{0}^{T}|\partial_{x}\varphi(X_{t}^{(k)})-\partial
_{x}\varphi(X_{t})|^{2}dt\rightarrow0\text{, as }k\rightarrow \infty.
\]
Thus we can apply Proposition \ref{Gt14} to prove that $\partial_{x}%
\varphi(X^{(k)})\beta^{(k)}\rightarrow$ $\partial_{x}\varphi(X)\beta$ in
$M_{\ast}^{2}(0,T)$. Similarly, $\partial_{x}\varphi(X^{(k)})\alpha
^{(k)}\rightarrow \partial_{x}\varphi(X)\alpha$, $\partial_{x}\varphi
(X^{(k)})\eta^{(k)}\rightarrow \partial_{x}\varphi(X)\eta$ and $\partial
_{xx}^{2}\varphi(X^{(k)})(\beta^{(k)})^{2}\rightarrow \partial_{xx}^{2}%
\varphi(X)\beta^{2}$ in $M_{\ast}^{1}(0,T)$. But from the above lemma we have
\begin{align*}
\varphi(X_{t}^{(k)})-\varphi(X_{0}^{(k)}) &  =\int_{0}^{t}\partial_{x}%
\varphi(X_{u}^{(k)})\beta_{u}^{(k)}dB_{u}+\int_{0}^{t}\partial_{x}%
\varphi(X_{u}^{(k)})\alpha_{u}^{(k)}du\\
&  +\int_{0}^{t}[\partial_{x}\varphi(X_{u}^{(k)})\eta_{u}^{(k)}+\frac{1}%
{2}\partial_{xx}^{2}\varphi(X_{u}^{(k)})(\beta_{u}^{(k)})^{2}]d\langle
B\rangle_{u}.
\end{align*}
We then can pass to the limit on both sides of the above equality, as
$k\rightarrow \infty$, to obtain (\ref{eq5.1}).
\end{proof}

\begin{lemma}
\label{Lemma2}Let $X$ be given as the above lemma and let $\varphi \in
C^{1,2}([0,\infty)\times \mathbb{R}^{n})$ such that $\varphi$, $\partial
_{t}\varphi$, $\partial_{x}\varphi$ and $\partial_{xx}^{2}\varphi$ are bounded
and uniformly continuous on $[0,\infty)\times \mathbb{R}^{n}$. Then we have%
\begin{align*}
\varphi(t,X_{t})-\varphi(0,X_{0}) &  =\int_{0}^{t}\partial_{x_{i}}%
\varphi(u,X_{u})\beta_{u}^{i}dB_{u}+\int_{0}^{t}[\partial_{t}\varphi
(u,X_{u})+(\partial_{x_{i}}\varphi(X_{u})\alpha_{u}^{i}]du\\
&  +\int_{0}^{t}[\partial_{x_{i}}\varphi(X_{u})\eta_{u}^{i}+\frac{1}%
{2}\partial_{x_{i}x_{j}}^{2}\varphi(X_{u})\beta_{u}^{i}\beta_{u}^{j}]d\langle
B\rangle_{u}.
\end{align*}

\end{lemma}

\begin{proof}
We can take $\{ \varphi_{k}\}_{k=1}^{\infty}$ such that, for each $k$,
$\varphi_{k}$ and all its first order and second order derivatives are in
$C_{b,Lip}^{2,2}((-\infty,\infty)\times \mathbb{R}^{n})$ and such that, as
$n\rightarrow \infty$, $\varphi_{n}$, $\partial_{t}\varphi_{n}$, $\partial
_{x}\varphi_{n}$ and $\partial_{xx}^{2}\varphi_{n}$ converge respectively to
$\varphi$, $\partial_{t}\varphi$, $\partial_{x}\varphi$ and $\partial_{xx}%
^{2}\varphi$ uniformly on $[0,\infty)\times \mathbb{R}$. We then use the above
It\^{o}'s formula to $\varphi_{n}(X_{t}^{0},X_{t})$, with $Y_{t}=(X_{t}%
^{0},X_{t})$, wiith $X_{t}^{0}\equiv t$:%
\begin{align*}
\varphi_{k}(t,X_{t})-\varphi_{k}(0,X_{0}) &  =\int_{0}^{t}\partial_{x_{i}%
}\varphi_{k}(u,X_{u})\beta_{u}^{i}dB_{u}+\int_{0}^{t}[\partial_{t}\varphi
_{k}(u,X_{u})+(\partial_{x_{i}}\varphi_{k}(u,X_{u})\alpha_{u}^{i}]du\\
&  +\int_{0}^{t}[\partial_{x_{i}}\varphi_{k}(u,X_{u})\eta_{u}^{i}+\frac{1}%
{2}\partial_{x_{i}x_{j}}^{2}\varphi_{k}(u,X_{u})\beta_{u}^{i}\beta_{u}%
^{j}]d\langle B\rangle_{u}.
\end{align*}
It follows that, as $k\rightarrow \infty$, we have uniformly
\begin{align*}
|\partial_{x_{i}}\varphi_{k}(u,X_{u})-\partial_{x_{i}}\varphi(u,X_{u})|  &
\rightarrow0\text{, }|\partial_{x_{i}x_{j}}^{2}\varphi_{k}(u,X_{u}%
)-\partial_{x_{i}x_{j}}^{2}\varphi_{k}(u,X_{u})|\rightarrow0\text{,}\\
|\partial_{t}\varphi_{k}(u,X_{u})-\partial_{t}\varphi(u,X_{u})|  &
\rightarrow0\text{.  }%
\end{align*}
We then can apply the above Lemma to $\varphi_{k}(t,X_{t})-\varphi_{k}%
(0,X_{0})$ and pass to the limit as $k\rightarrow \infty$ to obtain the desired reslut.
\end{proof}

\begin{theorem}
Let $\varphi \in C^{1,2}([0,\infty)\times \mathbb{R})$ and $X_{t}=X_{0}+\int
_{0}^{t}\alpha_{s}ds+\int_{0}^{t}\eta_{s}d\langle B\rangle_{s}+\int_{0}%
^{t}\beta_{s}dB_{s}$, where $\alpha,\eta$ in $M_{\omega}^{1}(0,T)$ and
$\beta \in M_{\omega}^{2}(0,T)$. Then for each $t\geq0$, we have
\begin{align*}
\varphi(t,X_{t})-\varphi(0,X_{0}) &  =\int_{0}^{t}\partial_{x_{i}}%
\varphi(u,X_{u})\beta_{u}^{i}dB_{u}+\int_{0}^{t}[\partial_{t}\varphi
(u,X_{u})+\partial_{x_{i}}\varphi(u,X_{u})\alpha_{u}^{i}]du\\
&  +\int_{0}^{t}[\partial_{x_{i}}\varphi(u,X_{u})\eta_{u}^{i}+\frac{1}%
{2}\partial_{x_{i}x_{j}}^{2}\varphi(u,X_{u})\beta_{u}^{i}\beta_{u}%
^{j}]d\langle B\rangle_{u}.
\end{align*}

\end{theorem}

\begin{proof}
We set, for $k=1,2,\cdots,$%
\[
\gamma_{t}=|X_{t}-X_{0}|+\int_{0}^{t}(|\beta_{u}|^{2}+|\alpha_{u}|+|\eta
_{u}|)du
\]
and $\tau_{k}:=\inf \{t\geq0|\gamma_{t}>k\} \wedge \sigma_{k}$. Let $\varphi_{k}$
be a $C^{1,2}$-function on $[0,\infty)\times \mathbb{R}^{n}$ such that
$\varphi$, $\partial_{t}\varphi$, $\partial_{x_{i}}\varphi$ and $\partial
_{x_{i}x_{j}}^{2}\varphi$ are uniformly bounded and such that $\varphi
_{k}=\varphi$, for $|x|\leq2k$, $t\in \lbrack0,T]$. It is clear that
\[
\mathbf{I}_{[0,\tau_{k}]}\beta \in M_{\ast}^{2}(0,T),\  \  \mathbf{I}%
_{[0,\tau_{k}]}\alpha,\  \  \mathbf{I}_{[0,\tau_{k}]}\eta \in M_{\ast}%
^{1}(0,T)\text{ }%
\]
and we have%
\[
X_{t\wedge \tau_{k}}^{i}=X_{0}^{i}+\int_{0}^{t}\alpha_{s}^{i}\mathbf{I}%
_{[0,\tau_{k}]}ds+\int_{0}^{t}\eta_{s}^{i}\mathbf{I}_{[0,\tau_{k}]}d\langle
B\rangle_{s}+\int_{0}^{t}\beta_{s}^{i}\mathbf{I}_{[0,\tau_{k}]}dB_{s}%
\]
We then can apply the above lemma to $\varphi_{k}(s,X_{s\wedge \tau_{k}})$,
$s\in \lbrack0,t]$ to obtain
\begin{align*}
\varphi(t,X_{t\wedge \tau_{k}})-\varphi(0,X_{0}) &  =\int_{0}^{t}%
\partial_{x_{i}}\varphi(u,X_{u})\beta_{u}^{i}\mathbf{I}_{[0,\tau_{k}]}%
dB_{u}+\int_{0}^{t}[\partial_{t}\varphi(u,X_{u})+\partial_{x_{i}}%
\varphi(u,X_{u})\alpha_{u}^{i}]\mathbf{I}_{[0,\tau_{k}]}du\\
&  +\int_{0}^{t}[\partial_{x_{i}}\varphi(u,X_{u})\eta_{u}^{i}\mathbf{I}%
_{[0,\tau_{k}]}+\frac{1}{2}\partial_{x_{i}x_{j}}^{2}\varphi(u,X_{u})\beta
_{u}^{i}\beta_{u}^{j}\mathbf{I}_{[0,\tau_{k}]}]d\langle B\rangle_{u}.
\end{align*}
Passing to the limit as $k\rightarrow \infty$ and applying Corollary
\ref{Coro4.8}, we then obtain the desired result.
\end{proof}

\begin{example}
\label{Exam5.5}For a given $\varphi \in C^{2}(\mathbb{R})$ we have%
\[
\varphi(B_{t})-\varphi(B_{t_{0}})=\int_{t_{0}}^{t}\varphi_{x}(B_{s}%
)dB_{s}+\frac{1}{2}\int_{0}^{t}\varphi_{xx}(B_{s})d\left \langle B\right \rangle
_{s}.
\]

\end{example}

\section{Appendix: Proof of Lemma \ref{lemma1}}

The proof is of Lemma \ref{lemma1} is very similar to those of Lemma 46 and
proposition 48 in Peng \cite{Peng2006b}  (see also \cite{Peng2007}). {W{{e
first consider the following simple case}}}.{{{ }}}

\begin{lemma}
\label{d-Lem-26}Let $\Phi \in C^{2}(\mathbb{R}^{n})$ with $\partial_{x^{v}}%
\Phi,\  \partial_{x^{\mu}x^{v}}^{2}\Phi \in C_{b.Lip}(\mathbb{R}^{n})$ for
$\mu,v=1,\cdots,n$. Let $s\in \lbrack0,T]$ be fixed and let $X=(X^{1}%
,\cdots,X^{n})^{T}$ be an $n$--dimensional process on $[s,T]$ of the form
\[
X_{t}^{j}=X_{s}^{j}+\alpha^{j}(t-s)+\eta^{j}(\left \langle B\right \rangle
_{t}-\left \langle B\right \rangle _{s})+\beta^{j}(B_{t}-B_{s}),
\]
where, for $j=1,\cdots,n$, $\alpha^{j}$, $\eta^{j}$ and $\beta^{j}$ are
bounded elements in $L_{\ast}^{2}(\Omega_{s})$ and $X_{s}=(X_{s}^{1}%
,\cdots,X_{s}^{n})^{T}$ is a given random vector in $L_{\ast}^{2}(\Omega_{s}%
)$. Then we have, in $L_{G}^{2}(\Omega_{t})$
\begin{align}
\Phi(X_{t})-\Phi(X_{s}) &  =\int_{s}^{t}\partial_{x^{j}}\Phi(X_{u})\beta
^{j}dB_{u}+\int_{s}^{t}\partial_{x^{j}}\Phi(X_{u})\alpha^{j}du\label{d-B-Ito}%
\\
&  +\int_{s}^{t}[\partial_{x^{j}}\Phi(X_{u})\eta^{j}+\frac{1}{2}%
\partial_{x^{i}x^{j}}^{2}\Phi(X_{u})\beta^{i}\beta^{j}]d\left \langle
B\right \rangle _{u}.\nonumber
\end{align}

\end{lemma}

\begin{proof}
For each positive integer $N$ we set $\delta=(t-s)/N$ and take the partition
\[
\pi_{\lbrack s,t]}^{N}=\{t_{0}^{N},t_{1}^{N},\cdots,t_{N}^{N}\}=\{s,s+\delta
,\cdots,s+N\delta=t\}.
\]
We have
\begin{align}
\Phi(X_{t})-\Phi(X_{s}) &  =\sum_{k=0}^{N-1}[\Phi(X_{t_{k+1}^{N}}%
)-\Phi(X_{t_{k}^{N}})]\label{d-Ito}\\
&  =\sum_{k=0}^{N-1}\{ \partial_{x^{j}}\Phi(X_{t_{k}^{N}})(X_{t_{k+1}^{N}}%
^{j}-X_{t_{k}^{N}}^{j})\nonumber \\
&  +\frac{1}{2}[\partial_{x^{i}x^{j}}^{2}\Phi(X_{t_{k}^{N}})(X_{t_{k+1}^{N}%
}^{i}-X_{t_{k}^{N}}^{i})(X_{t_{k+1}^{N}}^{j}-X_{t_{k}^{N}}^{j})+\eta_{k}%
^{N}]\},\nonumber
\end{align}
where
\[
\eta_{k}^{N}=[\partial_{x^{i}x^{j}}^{2}\Phi(X_{t_{k}^{N}}+\theta
_{k}(X_{t_{k+1}^{N}}-X_{t_{k}^{N}}))-\partial_{x^{i}x^{j}}^{2}\Phi
(X_{t_{k}^{N}})](X_{t_{k+1}^{N}}^{i}-X_{t_{k}^{N}}^{i})(X_{t_{k+1}^{N}}%
^{j}-X_{t_{k}^{N}}^{j})
\]
with $\theta_{k}\in \lbrack0,1]$. We have%
\begin{align*}
\mathbb{\hat{E}}[|\eta_{k}^{N}|^{2}] &  =\mathbb{\hat{E}}[|[\partial
_{x^{i}x^{j}}^{2}\Phi(X_{t_{k}^{N}}+\theta_{k}(X_{t_{k+1}^{N}}-X_{t_{k}^{N}%
}))-\partial_{x^{i}x^{j}}^{2}\Phi(X_{t_{k}^{N}})]\\
&  \times(X_{t_{k+1}^{N}}^{i}-X_{t_{k}^{N}}^{i})(X_{t_{k+1}^{N}}^{j}%
-X_{t_{k}^{N}}^{j})|^{2}]\\
&  \leq c\mathbb{\hat{E}[}|X_{t_{k+1}^{N}}-X_{t_{k}^{N}}|^{6}]\leq
C[\delta^{6}+\delta^{3}],
\end{align*}
where $c$ is the Lipschitz constant of $\{ \partial_{x^{i}x^{j}}^{2}%
\Phi \}_{i,j=1}^{n}${ }and $C$ is a constant independent of $k${{{.}}} Thus
\[
\mathbb{\hat{E}}[|\sum_{k=0}^{N-1}\eta_{k}^{N}|^{2}]\leq N\sum_{k=0}%
^{N-1}\mathbb{\hat{E}}[|\eta_{k}^{N}|^{2}]\rightarrow0.
\]
The rest terms in the summation of the right side of (\ref{d-Ito}) are
$\xi_{t}^{N}+\zeta_{t}^{N}$ with%
\begin{align*}
\xi_{t}^{N} &  =\sum_{k=0}^{N-1}\{ \partial_{x^{j}}\Phi(X_{t_{k}^{N}}%
)[\alpha^{j}(t_{k+1}^{N}-t_{k}^{N})+\eta^{j}(\left \langle B\right \rangle
_{t_{k+1}^{N}}-\left \langle B\right \rangle _{t_{k}^{N}})\\
&  +\beta^{j}(B_{t_{k+1}^{N}}-B_{t_{k}^{N}})]+\frac{1}{2}\partial_{x^{i}x^{j}%
}^{2}\Phi(X_{t_{k}^{N}})\beta^{i}\beta^{j}(B_{t_{k+1}^{N}}-B_{t_{k}^{N}}%
)^{2}\}
\end{align*}
and
\begin{align*}
\zeta_{t}^{N} &  =\frac{1}{2}\sum_{k=0}^{N-1}\partial_{x^{i}x^{j}}^{2}%
\Phi(X_{t_{k}^{N}})\{[\alpha^{i}(t_{k+1}^{N}-t_{k}^{N})+\eta^{i}(\left \langle
B\right \rangle _{t_{k+1}^{N}}-\left \langle B\right \rangle _{t_{k}^{N}})]\\
&  \times \lbrack \alpha^{j}(t_{k+1}^{N}-t_{k}^{N})+\eta^{j}(\left \langle
B\right \rangle _{t_{k+1}^{N}}-\left \langle B\right \rangle _{t_{k}^{N}})]\\
&  +2[\alpha^{i}(t_{k+1}^{N}-t_{k}^{N})+\eta^{i}(\left \langle B\right \rangle
_{t_{k+1}^{N}}-\left \langle B\right \rangle _{t_{k}^{N}})]\beta^{j}%
(B_{t_{k+1}^{N}}-B_{t_{k}^{N}})\}.
\end{align*}
We observe that, for each $u\in \lbrack t_{k}^{N},t_{k+1}^{N})$
\begin{align*}
&  \mathbb{\hat{E}}[|\partial_{x^{j}}\Phi(X_{u})-\sum_{k=0}^{N-1}%
\partial_{x^{j}}\Phi(X_{t_{k}^{N}})\mathbf{I}_{[t_{k}^{N},t_{k+1}^{N}%
)}(u)|^{2}]\\
&  =\mathbb{\hat{E}}[|\partial_{x^{j}}\Phi(X_{u})-\partial_{x^{j}}%
\Phi(X_{t_{k}^{N}})|^{2}]\\
&  \leq c^{2}\mathbb{\hat{E}}[|X_{u}-X_{t_{k}^{N}}|^{2}]\leq C[\delta
+\delta^{2}],
\end{align*}
{{{where $c$ is the Lipschitz constant of }}}$\{ \partial_{x^{j}}\Phi
\}_{j=1}^{n}$ and $C$ is a constant independent of $k$. Thus $\sum_{k=0}%
^{N-1}\partial_{x^{j}}\Phi(X_{t_{k}^{N}})\mathbf{I}_{[t_{k}^{N},t_{k+1}^{N}%
)}(\cdot)$ tends to $\partial_{x^{j}}\Phi(X_{\cdot})$ in $M_{\ast}^{2}(0,T)$.
Similarly,
\[
\sum_{k=0}^{N-1}\partial_{x^{i}x^{j}}^{2}\Phi(X_{t_{k}^{N}})\mathbf{I}%
_{[t_{k}^{N},t_{k+1}^{N})}(\cdot)\rightarrow \partial_{x^{i}x^{j}}^{2}%
\Phi(X_{\cdot})\text{ in \ }M_{\ast}^{2}(0,T).
\]
Let $N\rightarrow \infty$, from Lemma \ref{Lemm4.6}, Proposition \ref{Prop4.7}
and Corollary \ref{Coro4.8} as well as the definitions of the integrations of
$dt$, $dB_{t}$ and $d\left \langle B\right \rangle _{t}$ the limit of $\xi
_{t}^{N}$ in $L_{\ast}^{2}(\Omega_{t})$ is just the right hand side of
(\ref{d-B-Ito}). By the next Remark we also have $\zeta_{t}^{N}\rightarrow0$
in $L_{\ast}^{2}(\Omega_{t})$. We then have proved (\ref{d-B-Ito}).
\end{proof}

\begin{remark}
In the proof of $\zeta_{t}^{N}\rightarrow0$ in $L_{\ast}^{2}(\Omega_{t})$, we
use the following estimates: for $\psi^{N}\in M_{b,0}(0,T)$ with $\psi_{t}%
^{N}=\sum_{k=0}^{N-1}\xi_{t_{k}}^{N}\mathbf{I}_{[t_{k}^{N},t_{k+1}^{N})}(t)$,
and $\pi_{T}^{N}=\{t_{0}^{N},\cdots,t_{N}^{N}\}$ such that $\lim
_{N\rightarrow \infty}\mu(\pi_{T}^{N})=0$ and $\mathbb{\hat{E}}[\sum
_{k=0}^{N-1}|\xi_{t_{k}}^{N}|^{2}(t_{k+1}^{N}-t_{k}^{N})]\leq C$, for all
$N=1,2,\cdots$, we have $\mathbb{\hat{E}}[|\sum_{k=0}^{N-1}\xi_{k}^{N}%
(t_{k+1}^{N}-t_{k}^{N})^{2}|^{2}]\rightarrow0$, and%
\begin{align*}
\mathbb{\hat{E}}[|\sum_{k=0}^{N-1}\xi_{k}^{N}(\left \langle B\right \rangle
_{t_{k+1}^{N}}-\left \langle B\right \rangle _{t_{k}^{N}})^{2}|^{2}]  &  \leq
C\mathbb{\hat{E}}[\sum_{k=0}^{N-1}|\xi_{k}^{N}|^{2}(\left \langle
B\right \rangle _{t_{k+1}^{N}}-\left \langle B\right \rangle _{t_{k}^{N}})^{3}]\\
&  \leq C\mathbb{\hat{E}}[\sum_{k=0}^{N-1}|\xi_{k}^{N}|^{2}\bar{\sigma}%
^{6}(t_{k+1}^{N}-t_{k}^{N})^{3}]\rightarrow0,
\end{align*}%
\begin{align*}
&  \mathbb{\hat{E}}[|\sum_{k=0}^{N-1}\xi_{k}^{N}(\left \langle B\right \rangle
_{t_{k+1}^{N}}-\left \langle B\right \rangle _{t_{k}^{N}})(t_{k+1}^{N}-t_{k}%
^{N})|^{2}]\\
&  \leq C\mathbb{\hat{E}}[\sum_{k=0}^{N-1}|\xi_{k}^{N}|^{2}(t_{k+1}^{N}%
-t_{k}^{N})(\left \langle B\right \rangle _{t_{k+1}^{N}}-\left \langle
B\right \rangle _{t_{k}^{N}})^{2}]\\
&  \leq C\mathbb{\hat{E}}[\sum_{k=0}^{N-1}|\xi_{k}^{N}|^{2}\bar{\sigma}%
^{4}(t_{k+1}^{N}-t_{k}^{N})^{3}\rightarrow0,
\end{align*}
as well as
\begin{align*}
\mathbb{\hat{E}}[|\sum_{k=0}^{N-1}\xi_{k}^{N}(t_{k+1}^{N}-t_{k}^{N}%
)(B_{t_{k+1}^{N}}-B_{t_{k}^{N}})|^{2}]  &  \leq C\mathbb{\hat{E}}[\sum
_{k=0}^{N-1}|\xi_{k}^{N}|^{2}(t_{k+1}^{N}-t_{k}^{N})|B_{t_{k+1}^{N}}%
-B_{t_{k}^{N}}|^{2}]\\
&  \leq C\mathbb{\hat{E}}[\sum_{k=0}^{N-1}|\xi_{k}^{N}|^{2}\bar{\sigma}%
^{2}(t_{k+1}^{N}-t_{k}^{N})^{2}]\rightarrow0\
\end{align*}
and%
\begin{align*}
&  \mathbb{\hat{E}}[|\sum_{k=0}^{N-1}\xi_{k}^{N}(\left \langle B\right \rangle
_{t_{k+1}^{N}}-\left \langle B\right \rangle _{t_{k}^{N}})(B_{t_{k+1}^{N}%
}-B_{t_{k}^{N}})|^{2}]\\
&  \leq C\mathbb{\hat{E}}[\sum_{k=0}^{N-1}|\xi_{k}^{N}|^{2}(\left \langle
B\right \rangle _{t_{k+1}^{N}}-\left \langle B\right \rangle _{t_{k}^{N}%
})|B_{t_{k+1}^{N}}-B_{t_{k}^{N}}|^{2}]\\
&  \leq C\mathbb{\hat{E}}[\sum_{k=0}^{N-1}|\xi_{k}^{N}|^{2}\bar{\sigma}%
^{4}(t_{k+1}^{N}-t_{k}^{N})^{2}]\rightarrow0.
\end{align*}
\endproof

\end{remark}

{ { { We now give the proof of Lemma \ref{lemma1}:}}}

\begin{proof}
[Proof of Lemma 5.1]Let $\Phi \in C^{2}(\mathbb{R}^{n})$ with $\partial_{x^{j}%
}\Phi,\  \partial_{x^{i}x^{j}}^{2}\Phi \in C_{b.Lip}(\mathbb{R}^{n})$ for
$i,j=1,\cdots,n$. Let $\alpha^{j}$, $\beta^{j}$ and $\eta^{j}$, $j=1,\cdots
,n$, be bounded processes in $M_{\ast}^{2}(0,T)$. We need to prove that%
\begin{align}
\Phi(X_{t})-\Phi(X_{s}) &  =\int_{s}^{t}\partial_{x^{j}}\Phi(X_{u})\beta
_{u}^{j}dB_{u}+\int_{s}^{t}\partial_{x^{j}}\Phi(X_{u})\alpha_{u}%
^{j}du\label{d-Ito-form1}\\
&  +\int_{s}^{t}[\partial_{x^{j}}\Phi(X_{u})\eta_{u}^{j}+\frac{1}{2}%
\partial_{x^{i}x^{j}}^{2}\Phi(X_{u})\beta_{u}^{i}\beta_{u}^{j}]d\left \langle
B\right \rangle _{u}.\nonumber
\end{align}
We first consider the case where $\alpha$, $\eta$ and $\beta$ are step
processes of the form%
\[
\eta_{t}(\omega)=\sum_{k=0}^{N-1}\xi_{k}(\omega)\mathbf{I}_{[t_{k},t_{k+1}%
)}(t).
\]
From the above Lemma, it is clear that (\ref{d-Ito-form1}) holds true. Now
let
\[
X_{t}^{j,N}=X_{0}^{j}+\int_{0}^{t}\alpha_{s}^{j,N}ds+\int_{0}^{t}\eta
_{s}^{j,N}d\left \langle B\right \rangle _{s}+\int_{0}^{t}\beta_{s}^{j,N}dB_{s},
\]
where $\alpha^{N}$, $\eta^{N}$ and $\beta^{N}$ are uniformly bounded step
processes that converge to $\alpha$, $\eta$ and $\beta$ in $M_{\ast}^{2}(0,T)$
as $N\rightarrow \infty$. From Lemma \ref{d-Lem-26}%
\begin{align}
\Phi(X_{t}^{N})-\Phi(X_{s}^{N}) &  =\int_{s}^{t}\partial_{x^{j}}\Phi(X_{u}%
^{N})\beta_{u}^{j,N}dB_{u}+\int_{s}^{t}\partial_{x^{j}}\Phi(X_{u}^{N}%
)\alpha_{u}^{j,N}du\label{d-N-Ito}\\
&  +\int_{s}^{t}[\partial_{x^{j}}\Phi(X_{u}^{N})\eta_{u}^{j,N}+\frac{1}%
{2}\partial_{x^{i}x^{j}}^{2}\Phi(X_{u}^{N})\beta_{u}^{i,N}\beta_{u}%
^{j,N}]d\left \langle B\right \rangle _{u}.\nonumber
\end{align}
Since%
\begin{align*}
&  \mathbb{\hat{E}[}|X_{t}^{j,N}-X_{t}^{j}|^{2}]\\
&  \leq C\mathbb{\hat{E}[}\int_{0}^{T}[(\alpha_{s}^{j,N}-\alpha_{s}^{j}%
)^{2}+|\eta_{s}^{j,N}-\eta_{s}^{j}|^{2}+|\beta_{s}^{j,N}-\beta_{s}^{j}%
|^{2}]ds],
\end{align*}
where $C$ is a constant independent of $N$. We then can prove that, in
$M_{\ast}^{2}(0,T)$,
\begin{align*}
\partial_{x^{j}}\Phi(X_{\cdot}^{N})\eta_{\cdot}^{j,N} &  \rightarrow
\partial_{x^{j}}\Phi(X_{\cdot})\eta_{\cdot}^{j},\\
\partial_{x^{i}x^{j}}^{2}\Phi(X_{\cdot}^{N})\beta_{\cdot}^{i,N}\beta_{\cdot
}^{j,N} &  \rightarrow \partial_{x^{i}x^{j}}^{2}\Phi(X_{\cdot})\beta_{\cdot
}^{i}\beta_{\cdot}^{j},\\
\partial_{x^{j}}\Phi(X_{\cdot}^{N})\alpha_{\cdot}^{j,N} &  \rightarrow
\partial_{x^{j}}\Phi(X_{\cdot})\alpha_{\cdot}^{j},\\
\partial_{x^{j}}\Phi(X_{\cdot}^{N})\beta_{\cdot}^{j,N} &  \rightarrow
\partial_{x^{j}}\Phi(X_{\cdot})\beta_{\cdot}^{j}.
\end{align*}
We then can pass limit in both sides of (\ref{d-N-Ito}) to get
(\ref{d-Ito-form1}).
\end{proof}

\end{document}